\documentclass[11pt]{amsart} 

\usepackage[margin=1in]{geometry}       
\usepackage[english]{babel}             
\usepackage{CJKutf8}                    
\usepackage[utf8]{inputenc}             
\usepackage{amsmath,amssymb,amsthm}     
\usepackage{graphicx}		            
\usepackage[style=alphabetic]{biblatex} 
\usepackage{csquotes}                   
\usepackage{multirow}                   
\usepackage{booktabs}                   
\usepackage{hyperref}                   
\usepackage[all]{hypcap}                
\newtheorem{theorem}{Theorem}
\newtheorem{proposition}{Proposition}
\newtheorem{corollary}{Corollary}[theorem]

\theoremstyle{definition}
\newtheorem{example}{Example}

\newtheorem{innermanprop}{Proposition}
\newenvironment{manprop}[1]
    {\renewcommand\theinnermanprop{#1}\innermanprop}
    {\endinnermanprop}
    
\newtheorem{innermanthm}{Theorem}
\newenvironment{manthm}[1]
    {\renewcommand\theinnermanthm{#1}\innermanthm}
    {\endinnermanthm}
    
\newtheorem{innermancor}{Corollary}
\newenvironment{mancor}[1]
    {\renewcommand\theinnermancor{#1}\innermancor}
    {\endinnermancor}

\DeclareMathOperator{\arf}{Arf}

\title{The delta-unlinking number of algebraically split links}
\author{Anthony Bosman$^\dagger$, Jeannelle Green$^\dagger$, Gabriel Palacios$^\dagger$, Moises Reyes$^\dagger$, Noe Reyes$^\dagger$}
\date{\today}
\address{Department of Mathematics, Andrews University, 4260 Administration Dr., Berrien Springs, MI 49104}
\email{bosman@andrews.edu}
\thanks{$^\dagger$Partially supported by National Science Foundation grant DMS-1950644.}
\begin{document}
\maketitle
\thispagestyle{empty}

\begin{abstract}
    It is known that algebraically split links (links with vanishing pairwise  linking number) can be transformed into the trivial link by a series of local moves on the link diagram called delta-moves; we define the delta-unlinking number to be the minimum number of such moves needed. This generalizes the notion of delta-unknotting number, defined to be the minimum number of delta-moves needed to move a knot into the unknot. While the delta-unknotting number has been well-studied and calculated for prime knots, no prior such analysis has been conducted for the delta-unlinking number. We prove a number of lower and upper bounds on the delta-unlinking number, relating it to classical link invariants including unlinking number, 4-genus, and Arf invariant. This allows us to determine the precise value of the delta-unlinking number for algebraically split prime links with up to 9 crossings as well as determine the 4-genus for most of these links. 
\end{abstract}

\section*{Acknowledgments}
This work arises from the National Research Experience for Undergraduates Program conducted at Andrews University during the summer of 2021. Special appreciation to the National Science Foundation, the Mathematical Association of America, and the Department of Mathematics of Andrews University for their generous investment.

\section{Introduction}

An $m$-component link is the isotopy class of an embedding of $\sqcup_m S^1 \longrightarrow S^3$; a knot is a 1-component link. A link can be depicted as a diagram representing its projection onto the plane. The {\it $\Delta$-move} is the local move on a link diagram that transforms the region within a disk as in Figure \ref{fig:delta_move} and leaves the rest of the diagram unchanged. The $\Delta$-move is known to be a unknotting move \cite{murakami1989certain}; therefore, every knot $K$ can be deformed into the unknot via some sequence of $\Delta$-moves. The {\it $\Delta$-unknotting number} $u^\Delta(K)$ is the minimal number of $\Delta$-moves needed to deform $K$ into the unknot; it has been calculated for prime knots with up to 10 crossings \cite{nakamura1998delta}.

In a link, we allow the three strands of the $\Delta$-move to belong to any component(s) of the link; in the case that they all belong to the same component, it is called a {\it self $\Delta$-move} as in Figure \ref{fig:delta_types}.

\begin{figure}[h]
    \centering
    \includegraphics[width=0.4\textwidth]{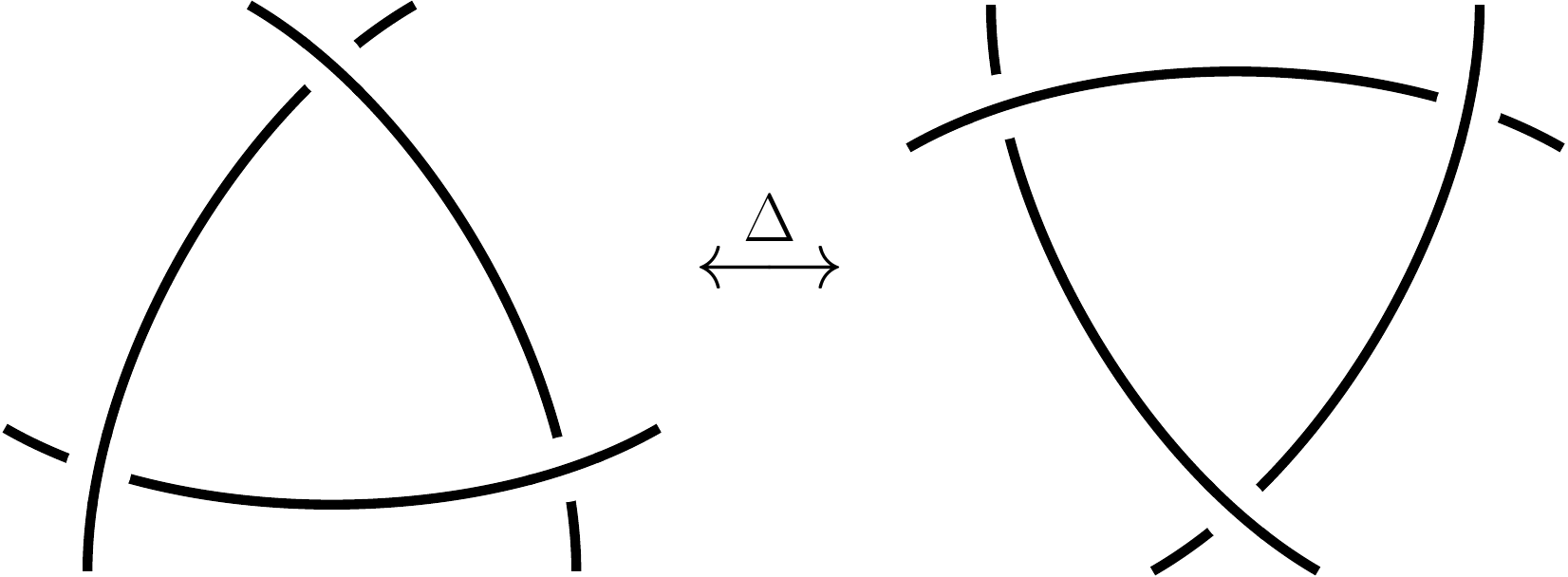}
    \caption{A $\Delta$-move.}
    \label{fig:delta_move}
\end{figure}

We call two links {\it (self) $\Delta$-equivalent} if one link can be deformed into the other by (self) $\Delta$-moves involving any number of components.

The $\Delta$-moves were first introduced in \cite{murakami1989certain}. In a link, it is not hard to see that a $\Delta$-move preserves linking number. Murakami and Nakanishi proved the converse, giving the following classification of links up to $\Delta$-equivalence:

\begin{theorem}[Theorem 1.1 in \cite{murakami1989certain}]
    Let $L = L_1 \sqcup L_2 \sqcup \cdots \sqcup L_m$ and $L' = L_1' \sqcup L_2' \sqcup \cdots \sqcup L_m'$ be ordered oriented $m$-component links. Then $L$ and $L'$ are $\Delta$-equivalent if and only if $lk(L_i,L_j) = lk(L'_i,L'_j)$ for $1 \leq i < j \leq m$, where $lk$ denotes the linking number.
\end{theorem}

Given $\Delta$-equivalent links $L$ and $L'$, we can define the {\it (self) $\Delta$-Gordian distance} $d_G^\Delta(L,L')$ between the links to be the minimal number of (self) $\Delta$-moves needed to deform one link into the other. In particular, if an $m$-component link $L$ is algebraically split, that is, if $L$ has vanishing pairwise linking numbers, then $L$ is $\Delta$-equivalent to the trivial link with $m$-components, denoted $0_1^m$. We denote the distance between an algebraically split link $L$ and the trivial link by $u^\Delta(L)$, the {\it $\Delta$-unlinking number}.

In the case of knots, it has been shown that $u^\Delta$ and $d^\Delta_G$ relate to many other well-known knot invariants including the unknotting number and Arf invariant \cite{murakami1989certain}. We generalize these relations to links (or in the case of the Arf invariant, to proper links) and relate $u^\Delta$ to other link invariants such as the 4-genus.

\noindent For instance, in Section \ref{sec:lower_bounds} we show:
\begin{manprop}{\ref{prop:half_unlinking_number}}
    Given an algebraically split link $L$,
    \[u^\Delta(L) \geq \frac{1}{2} u(L)\]
    where $u(L)$ denotes the unlinking number of $L$.
\end{manprop}

\noindent We also find a relationship with 4-genus:
\begin{manthm}{\ref{thm:4_genus_links_distance}}
    For $\Delta$-equivalent proper links $L, L'$, 
    \[d^\Delta_G(L,L') \geq |g_4(L) - g_4(L')|.\]
\end{manthm}

\noindent And hence:
\begin{mancor}{\ref{cor:4_genus_links_unlinking}}
    Given an algebraically split link $L$,
    \[u^{\Delta}(L) \geq g_4(L).\]
\end{mancor}

\noindent Moreover, we show the following in Section \ref{sec:other_methods}:
\begin{manthm}{\ref{thm:arf_links}}
    Given $\Delta$-equivalent proper links $L,L'$ we have \[d^\Delta_G(L,L') \equiv \arf(L) + \arf(L') \pmod{2}.\]
\end{manthm}

\noindent It then immediately follows:
\begin{mancor}{\ref{cor:arf}}
    Given an algebraically split link $L$,
    \[u^\Delta(L) \equiv \arf(L) \pmod{2}.\]
\end{mancor}

These and other such bounds allow us to determine $u^\Delta$ for algebraically split prime links up to 9 crossings; see Section \ref{sec:table}. We also determine the 4-genus for nearly all of these links.
\section{Lower bounds on $\Delta$-unlinking number}
\label{sec:lower_bounds}

\subsection{Unknotting number and unlinking number}

A $\Delta$-move is independent of the choice of orientation and mirroring \cite{murakami1989certain}. In particular, if $L$ and $L'$ are $\Delta$-equivalent then so are their mirrors $mL$ and $mL'$. Thus $u^\Delta(L)=u^\Delta(mL)$. The local moves in Figure \ref{fig:clasp_pass_moves} are equivalent to $\Delta$-moves \cite{taniyama2002clasp}; here the strands may belong to any component(s) of the link, except in the case of the self $\Delta$-move, in which case they must all belong to the same component. These alternative representations of the $\Delta$-move are instrumental for finding $\Delta$-pathways between $\Delta$-equivalent paths.

\begin{figure}
    \centering
    \includegraphics[width=0.8\textwidth]{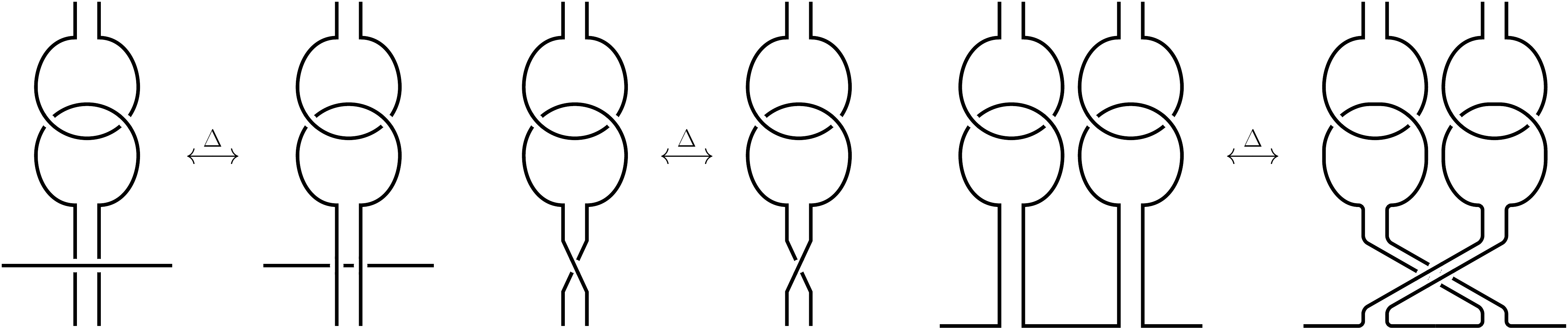}
    \caption{Moves equivalent to a $\Delta$-move.}
    \label{fig:clasp_pass_moves}
\end{figure}

Given links $L$ and $L'$, one may transform $L$ into $L'$ and then $L'$ into the trivial link, or vice versa. Thus $d^\Delta_G$ is a metric on a set of links with equivalent linking number. For algebraically split links, it then follows from the triangle inequality:
\[d^\Delta_G(L,L') \leq u^\Delta(L) + u^\Delta(L'), \hspace{0.8cm}
u^\Delta(L) \leq d^\Delta_G(L,L') + u^\Delta(L'), \hspace{0.8cm}
u^\Delta(L') \leq d^\Delta_G(L,L') + u^\Delta(L).\]

\noindent and so
\[|u^\Delta(L) - u^\Delta(L')| \leq d^\Delta_G(L,L') \leq u^\Delta(L) + u^\Delta(L').\]

\begin{figure}[h]
    \centering
    \includegraphics{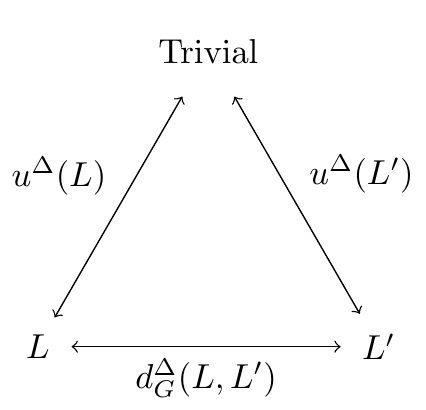}
    \caption{The $\Delta$-inequality.}
    \label{fig:triangle_inequality}
\end{figure}

A $\Delta$-move can be accomplished by two crossing changes; see Figure \ref{fig:delta_crossing}. Thus $d_G(L,L')\leq 2 d_G^\Delta(L,L')$ where $d_G(L,L')$ denotes the Gordian distance between $L$ and $L$, that is, the minimal number of crossing changes needed to transform $L$ into $L'$. It immediately follows:

\begin{proposition}
    \label{prop:half_unlinking_number}
    Given an algebraically split link $L$,
    \[u^\Delta(L) \geq \frac{1}{2} u(L)\]
    where $u(L)$ denotes the unlinking number of $L$.
\end{proposition}

\begin{figure}[b]
    \centering
    \includegraphics[width=0.8\textwidth]{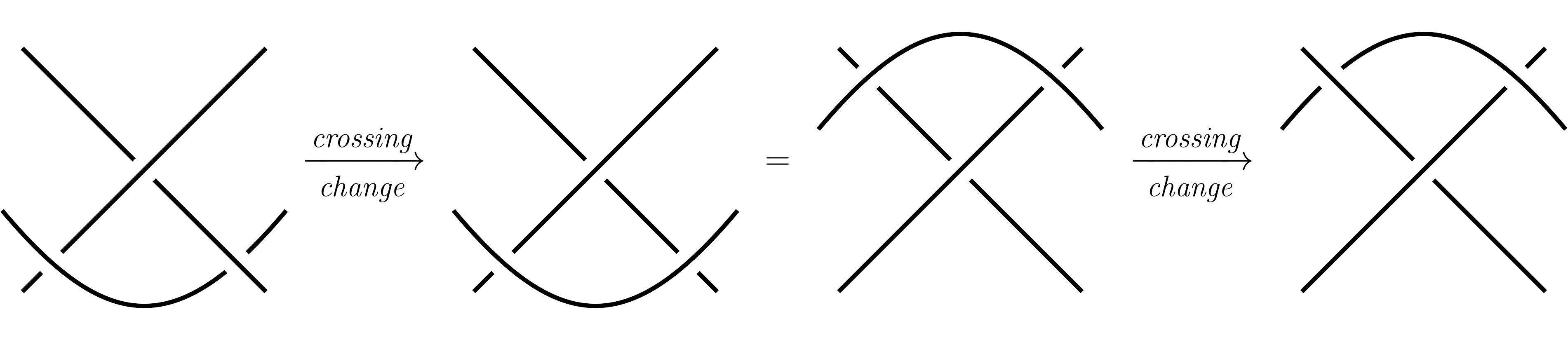}
    \caption{Two crossing changes are necessary to perform a $\Delta$-move.}
    \label{fig:delta_crossing}
\end{figure}

This lower bound may achieve the $\Delta$-unlinking number for a link. For instance, consider the link $L9a2$ (as seen in Figure \ref{fig:L9a2}). By \cite{nagel2015unlinking}, we know $u(L9a2) = 3$ and so $u^\Delta(L9a2) \geq 2$. Moreover, there exists a $\Delta$-pathway comprising only two $\Delta$-moves: a $\Delta$-move transforms $L9a2$ into the split union $3_1\#0_1$ which is again transformed by a $\Delta$-move into the 2-component trivial link. The inequality of Proposition \ref{prop:half_unlinking_number} may be strict; see table in Section \ref{sec:table}.

\begin{figure}[h]
    \centering
    \includegraphics[width=0.25\textwidth]{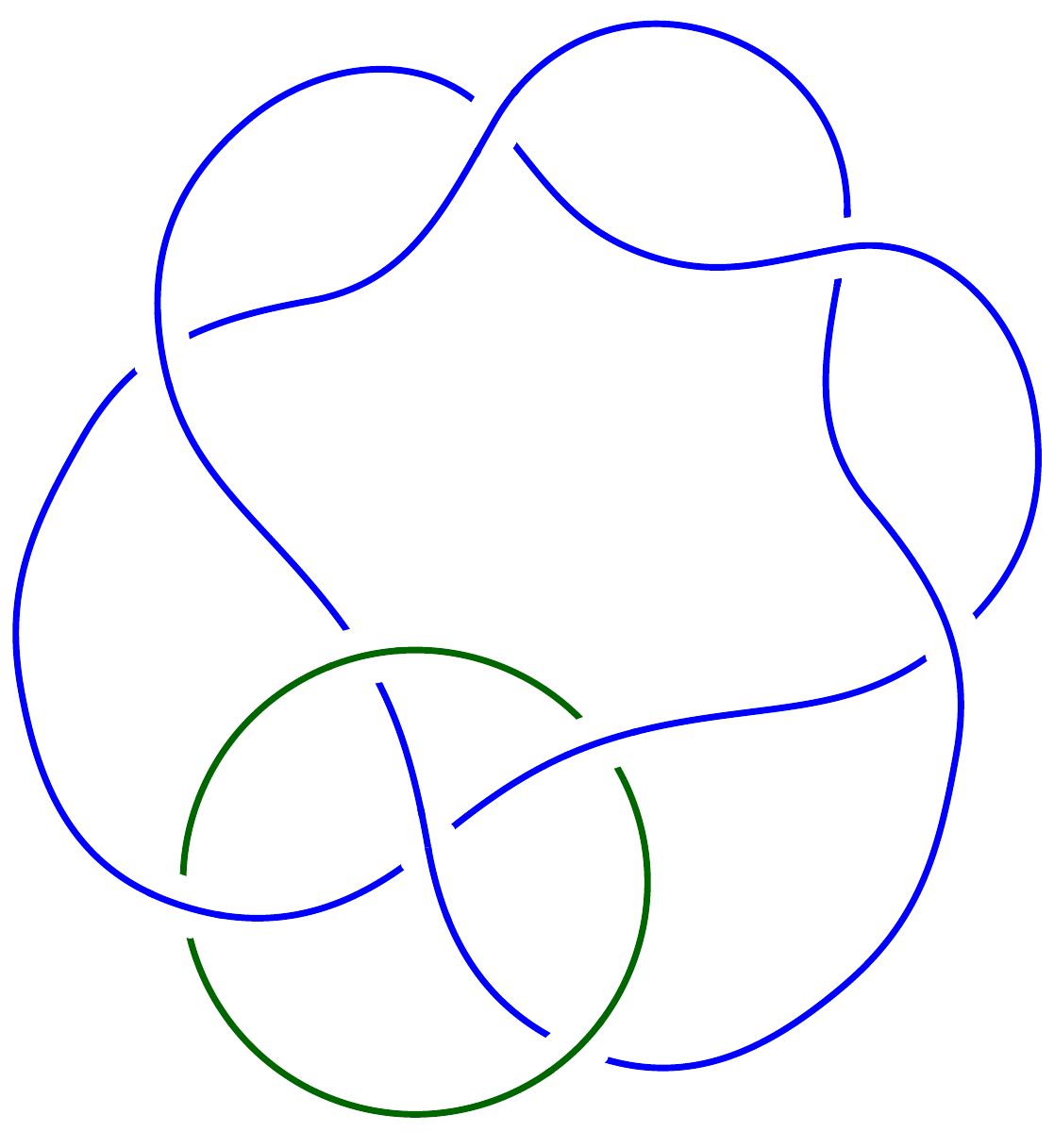}
    \caption{The link $L9a2$.}
    \label{fig:L9a2}
\end{figure}

\begin{proposition}
    \label{prop:self_delta_inequality}
    Given an algebraically split link $L = L_1 \sqcup L_2 \sqcup \cdots \sqcup L_m$, \[u^\Delta(L)\geq u^\Delta(L_1) + u^\Delta(L_2) + \cdots + u^\Delta(L_m).\]
    
    \noindent Moreover, if we have equality, then $L$ is self $\Delta$-equivalent to the trivial link.
\end{proposition}
\begin{proof}
    \begin{figure}[h]
        \centering  
        \includegraphics[width=0.32\textwidth]{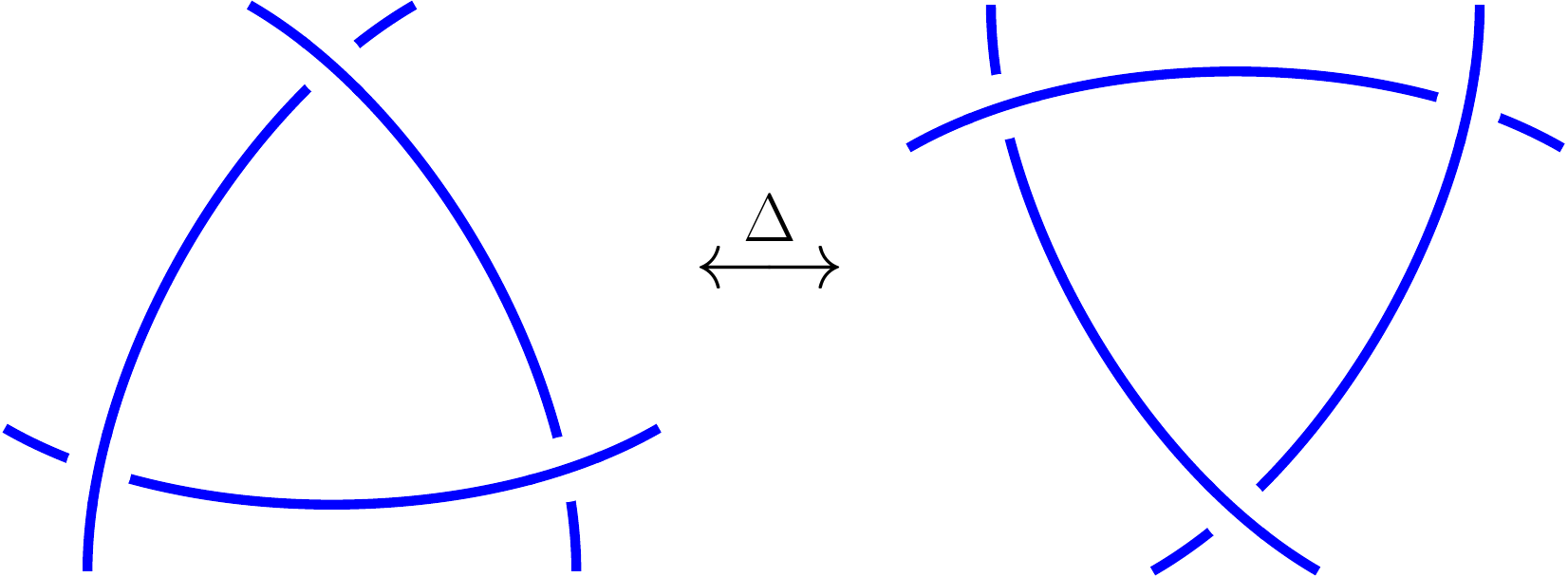}
        \includegraphics[width=0.32\textwidth]{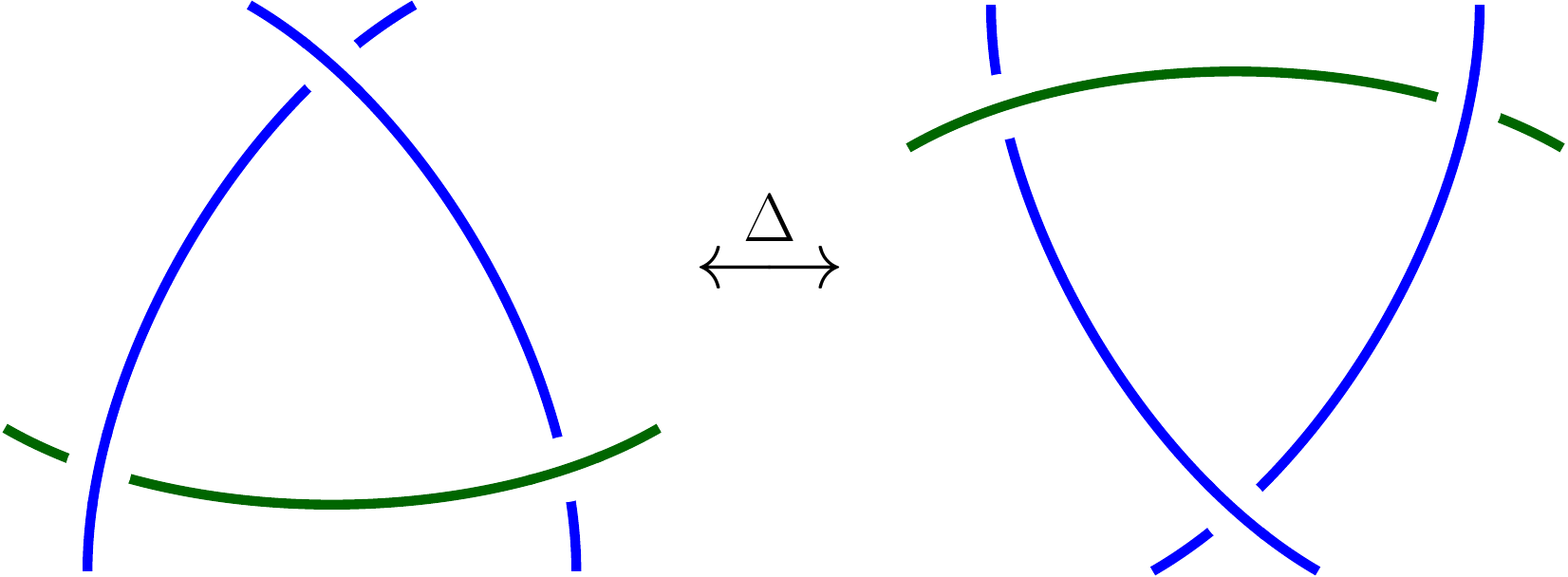}
        \includegraphics[width=0.32\textwidth]{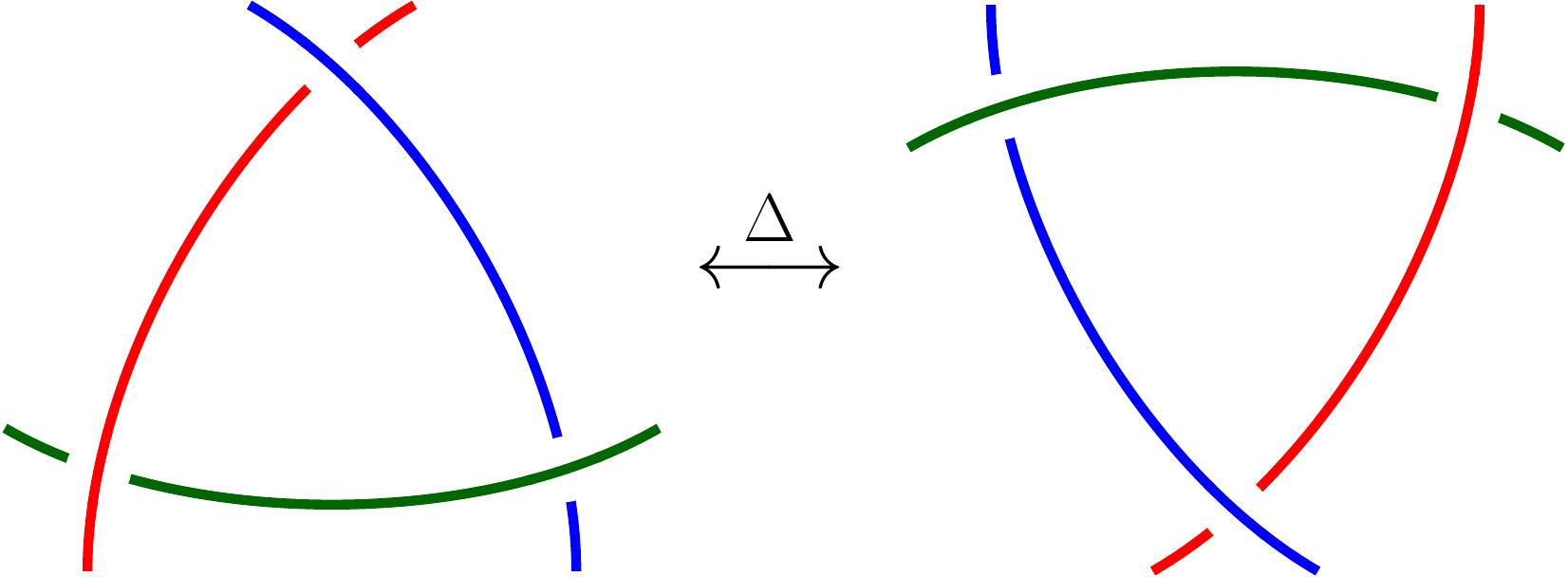}
        \caption{Left to right: A self $\Delta$-move, a $\Delta$-move involving two components, and a $\Delta$-move involving three components.}
        \label{fig:delta_types}
    \end{figure}
    
    Transforming the link $L$ into the trivial link requires unknotting each component. Note that only self $\Delta$-moves modify the knot type of any component of a link; see Figure \ref{fig:delta_types}. Moreover, a self $\Delta$-move only changes the knot type of a single component. The inequality follows.
    
    Now, suppose we have equality. Then unknotting the components with self $\Delta$-moves is sufficient to obtain the trivial link. Thus $L$ is self $\Delta$-equivalent to the trivial link.
\end{proof}

Observe, however, that the converse of the last part of the proposition fails. That is, there exist links which are self $\Delta$-equivalent to the trivial link that have $u^\Delta(L) > u^\Delta(L_1) + \cdots + u^\Delta(L_m)$. For instance, the Bing double of a knot is an algebraically split link and a boundary link \cite{cimasoni2006slicing} and thus by Corollary \ref{cor:self_delta_classification} it is self $\Delta$-equivalent to the trivial link, but the components of a Bing double are unknotted.

\subsection{4-genus}

Recall that the 4-genus $g_4(L)$ of a link $L = L_1 \sqcup L_2 \sqcup \cdots \sqcup L_m$ is defined as 
\[g_4(L)=\min\left(\sum_{i=1}^m g(F_i)\mid F_1\sqcup\cdots\sqcup F_m\hookrightarrow B^4,\; \partial F_i=L_i \right),\]
where the minimum is over smooth embeddings of the disjoint, oriented surfaces $F_1, F_2, \ldots, F_m$ in the 4-ball $B^4$. Meanwhile, the slice genus $g^*(L)$ is the minimal genus of a single such embedded surface that has $L$ as its boundary. In particular, $g_4(L)\geq g^*(L)$.

\begin{theorem}
    \label{thm:4_genus_links_distance}
    Given $\Delta$-equivalent links $L$ and $L'$, we have 
    \[d^\Delta_G(L,L') \geq |g_4(L) - g_4(L')|.\]
\end{theorem}
\begin{proof}
    Suppose $d^{\Delta}_G (L,L') = n$. Since a $\Delta$-move can be represented as fusion with the Borromean rings \cite{murakami1989certain} (see Figure \ref{fig:borro_rings_fusion}), $L$ is the result of fusion of the split union of $L'$ with $n$ copies of the Borromean rings $B_1, B_2, \ldots, B_n$. Note that each component $L_i$ of $L$ is fused with a component of the Borromean rings $B_j$ exactly when the arc in the corresponding $\Delta$-move belongs to $L_i$. Thus there exist embeddings of disjoint, oriented surfaces $F_1, \ldots, F_m$ in $S^3 \times [0,1]$ such that $F_i \cap (S^3 \times \{0\}) = L_i$ and  $F_i \cap (S^3 \times \{1\})$ is the fusion of $L_i'$ with the components of the Borromean rings that correspond to arcs of $\Delta$-moves belonging to $L'_i$.
    \begin{figure}
        \centering
        \includegraphics[width=0.5\textwidth]{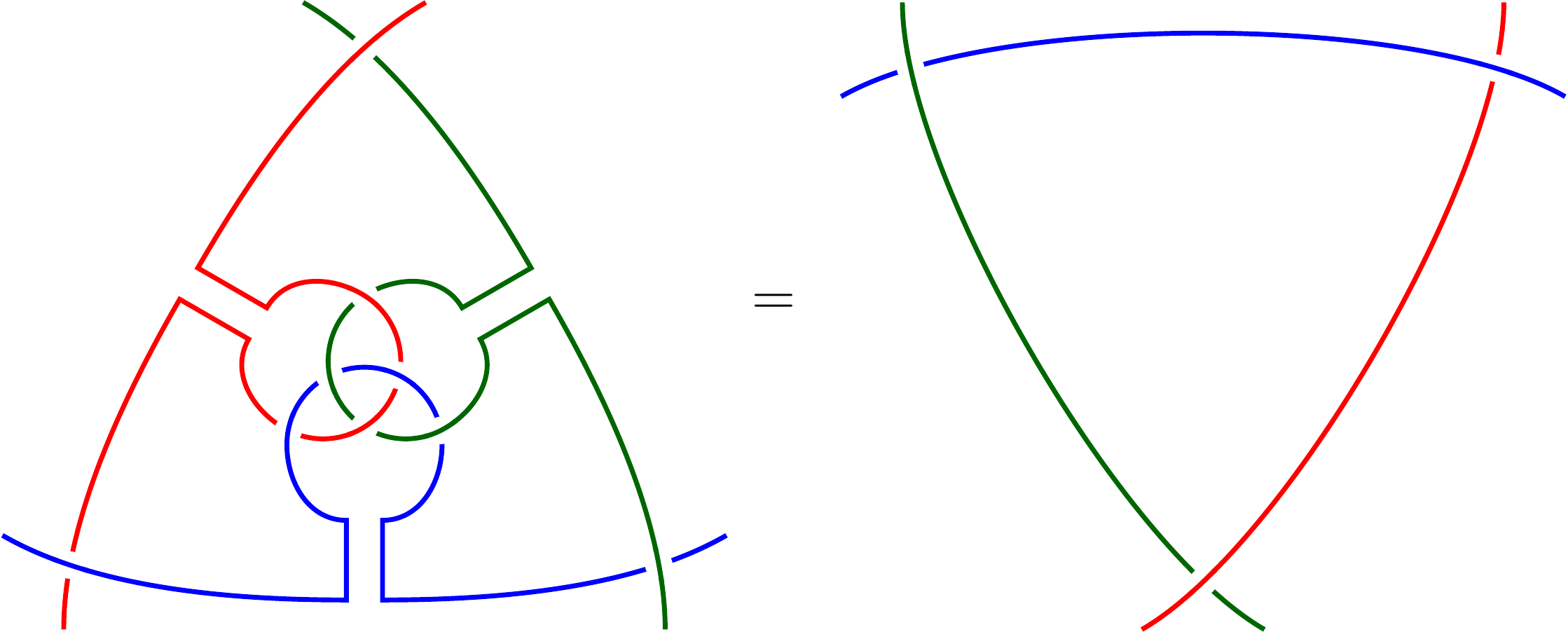}
        \caption{A $\Delta$-move achieved by fusion with the Borromean rings.}
        \label{fig:borro_rings_fusion}
    \end{figure}
    
    By fusing a component of the Borromean rings with itself, isotopying, then fusing the components back together, as in Figure \ref{fig:borro_rings_4g} (cf. \cite{sugishita1983triple}), we see that each $B_i$ bounds three disjoint surfaces: one surface of genus 1 and two disks. On our surfaces $F_1,\ldots,F_m$ we can thus cap off the components of $B_i$, contributing $n$ to the total genus. See Figure \ref{fig:disjoint_surfaces}. Also, we can cap off $L_1',\ldots,L_m'$ with disjoint surfaces that each have total genus $g_4(L')$.
    \begin{figure}[b]
        \centering
        \includegraphics[width=0.6\textwidth]{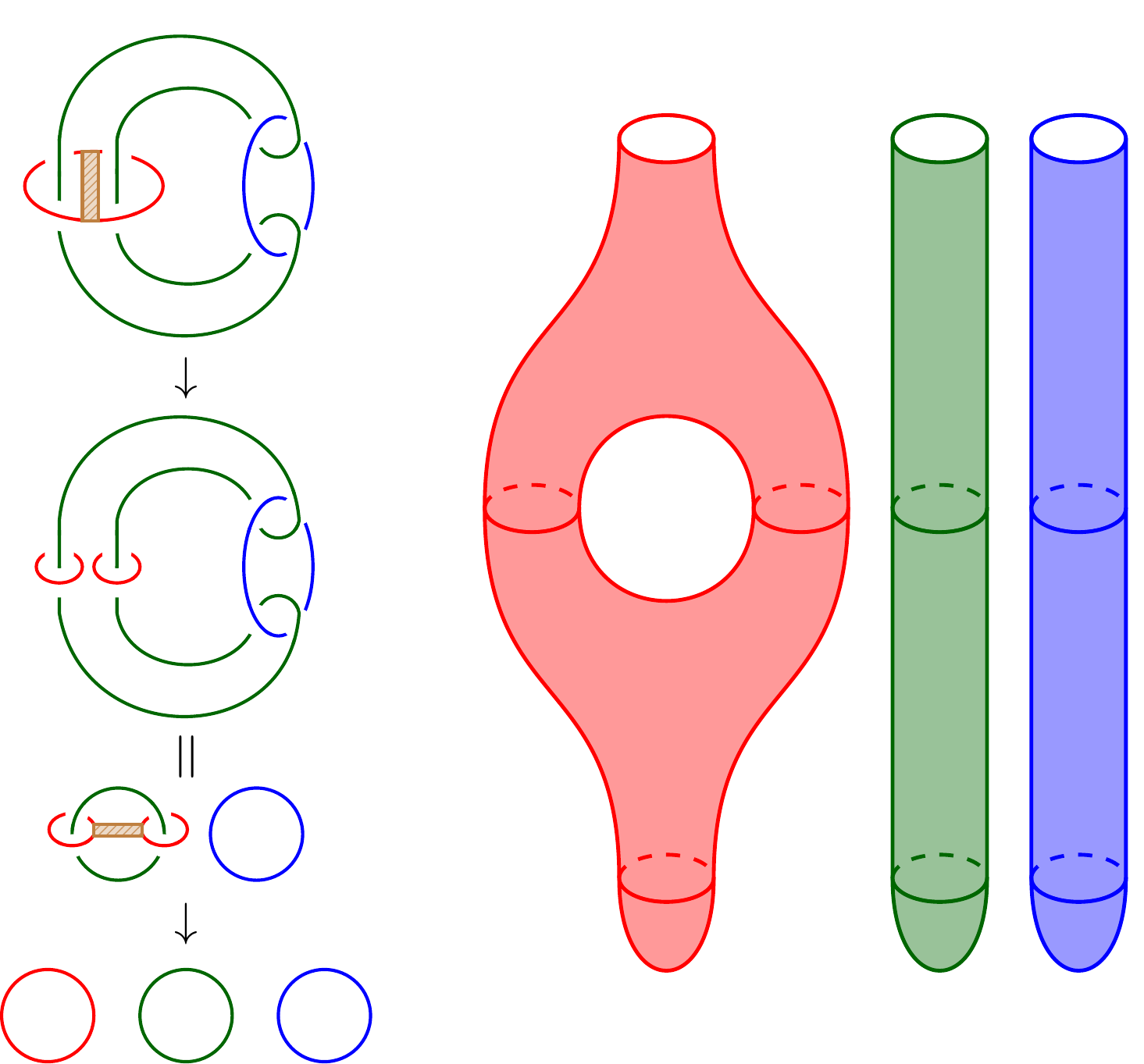}
        \caption{The Borromean rings have 4-genus of 1 from a genus 1 surface and two disks.}
        \label{fig:borro_rings_4g}
    \end{figure}
    
    Thus the components of $L$ bound disjoint, oriented surfaces with total genus $d^\Delta_G(L,L') + g_4(L')$, giving
    \[g_4(L) \leq d^\Delta_G(L,L') + g_4(L').\]
    
    \noindent Since $\Delta$-moves are reversible, by symmetry we similarly have,
    \[g_4(L') \leq d^\Delta_G(L,L') + g_4(L).\]
    
    \noindent The result follows.
\end{proof}

\noindent By letting $L'$ be the trivial link in Theorem \ref{thm:4_genus_links_distance}, we have the following corollary.
\begin{corollary}
    \label{cor:4_genus_links_unlinking}
    Given an algebraically split link $L$,
    \[u^{\Delta}(L) \geq g_4(L).\]
\end{corollary}

Note that the bound also holds in the topological category, since any smooth embedding of a surface is locally flat and hence $g_4^{top}(L) \leq g_4(L)$.

\begin{figure}[h]
    \centering
    \includegraphics[width=0.8\textwidth]{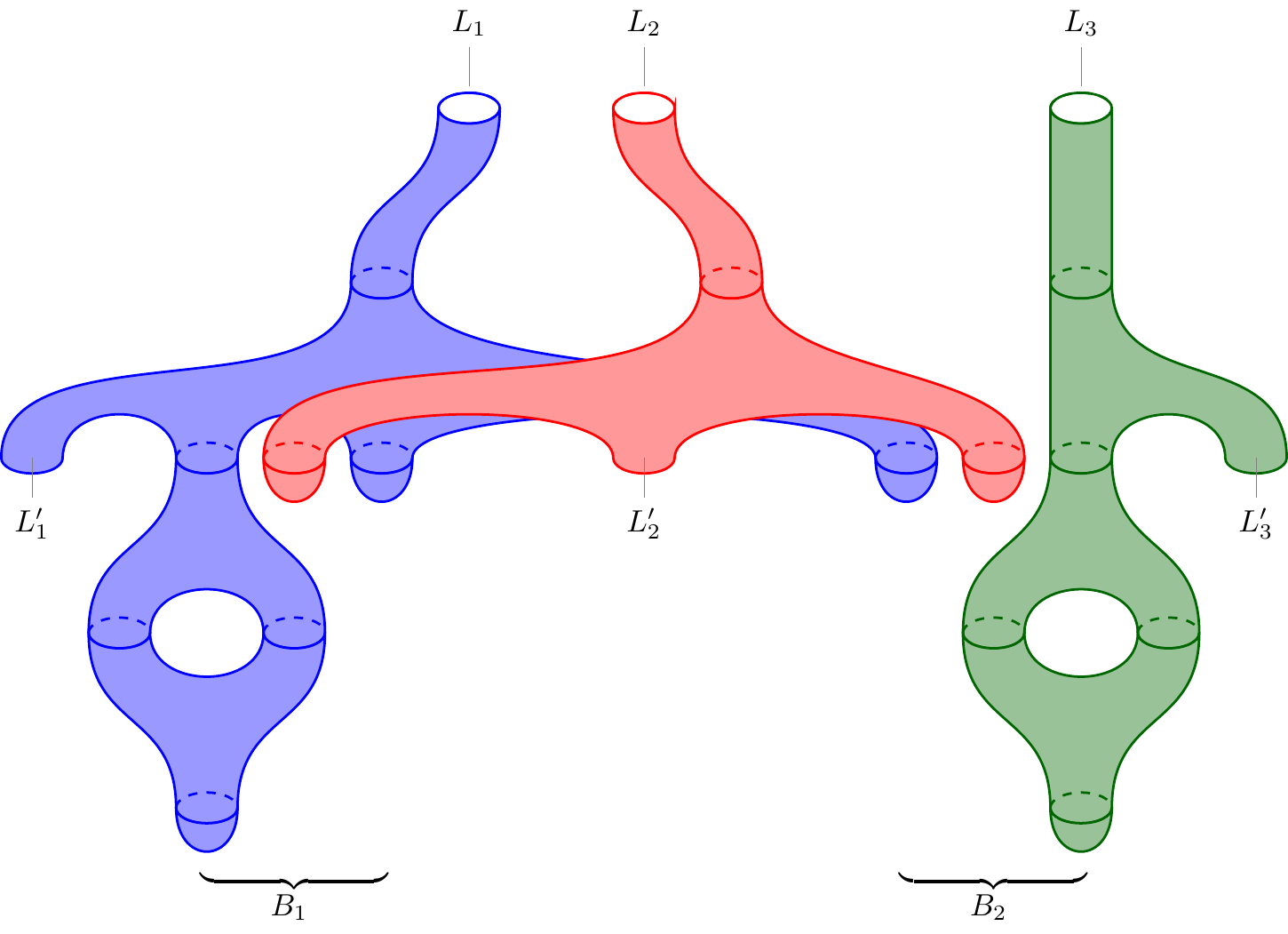}
    \caption{Capping off sets of Borromean rings, each contributing 1 to the 4-genus.}
    \label{fig:disjoint_surfaces}
\end{figure}
\section{Other methods for determining $\Delta$-unlinking number}
\label{sec:other_methods}

\subsection{Arf Invariant}

Recall, a link $L = L_1 \sqcup L_2 \sqcup \cdots \sqcup L_m$ is called a {\it proper link} if
\[\sum_{1 \leq i < j \leq m} lk(L_i,L_j) \equiv 0 \pmod{2}.\]

Robertello showed that the Arf invariant is well-defined for proper links $L$ \cite{hoste1984arf,robertello1965invariant}. In particular, if a proper link $L$ cobounds a planar surface with a knot $K$ then we may define $\arf(L) := \arf(K)$.

\begin{theorem}
    \label{thm:arf_links}
    Given $\Delta$-equivalent proper links $L$ and $L'$, we have \[d^\Delta_G(L,L') \equiv \arf(L) + \arf(L') \pmod{2}.\]
\end{theorem}
\begin{proof}
    Suppose $d^\Delta_G(L,L') = n$. Representing the $\Delta$-move as band fusion with the Borromean rings \cite{murakami1989certain} (see Figure \ref{fig:borro_rings_fusion}), $L'$ is the result of the fusion of the split union of $L$ with $n$ copies of the Borromean rings $B_1, B_2, \ldots, B_n$. Thus, for each component of $L$, we can construct disjoint surfaces $F_1, F_2, \ldots, F_m$ embedded in $S^3 \times [0,1]$. Now, we may fuse the components of $L'$ to obtain a knot $K$. But $L' = L \sqcup B_1 \sqcup B_2 \sqcup \cdots \sqcup B_n$. Thus, there exists a planar surface cobounded by $K$ and $L'$ and hence $\arf(K) = \arf(L') = \arf(L \sqcup B_1 \sqcup B_2 \sqcup \cdots \sqcup B_n)$. Then
        \[\arf(L') \equiv \arf(L) + \arf(B_1) + \arf(B_2) + \cdots + \arf(B_n) \pmod{2}. \]
    
    \noindent And since $\arf(B_i) = 1$, we conclude
    \[\arf(L') + \arf(L) \equiv n \pmod{2}.\qedhere\] 
\end{proof}

\noindent By letting $L'$ be the trivial link in Theorem \ref{thm:arf_links}, we obtain the following corollary.

\begin{corollary}
    \label{cor:arf}
    Given an algebraically split link $L$,
    \[u^\Delta(L) \equiv \arf(L) \pmod{2}.\]
\end{corollary}

\begin{example}
    The link $L9a40$ has $g_4(L9a40) = 2$. Thus by Corollary \ref{cor:4_genus_links_unlinking}, $u^\Delta(L9a40)\geq 2$; however, since $\arf(L9a40) = 1$, by Corollary \ref{cor:arf} we have $u^\Delta(L9a40)\geq 3$. In fact, there is a path of three $\Delta$-moves transforming $L9a40$ into the trivial link $0^2_1$: $L9a40 \overset{\Delta}{\longleftrightarrow} mL7a4 \overset{\Delta}{\longleftrightarrow} mL5a1 \overset{\Delta}{\longleftrightarrow} 0^2_1$. See Figure \ref{fig:L9a40_pathway}.
\end{example}

\begin{figure}[h]
    \centering
    \includegraphics[width=0.9\textwidth]{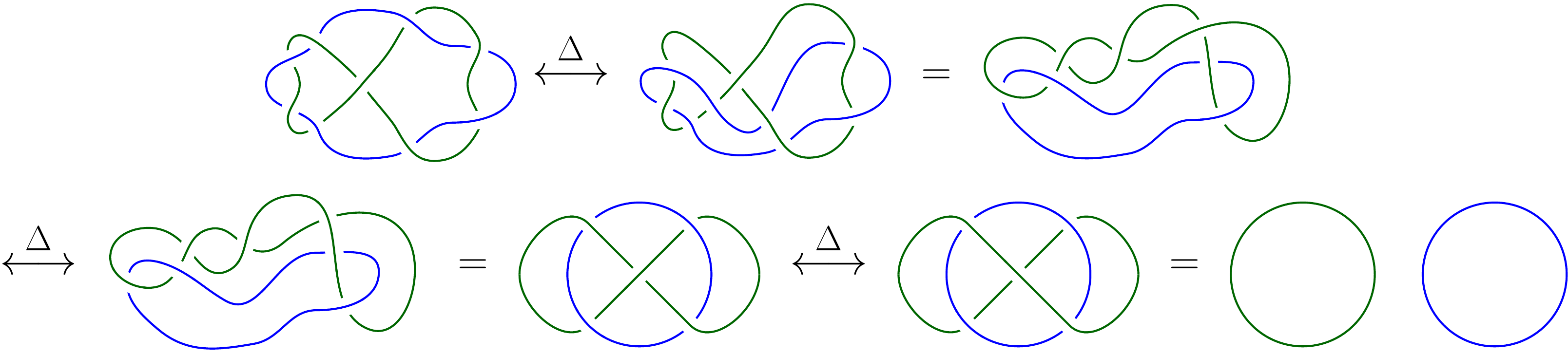}
    \caption{A sequence of $\Delta$-moves unlinking $L9a40$.}
    \label{fig:L9a40_pathway}
\end{figure}

Note it immediately follows from Corollary \ref{cor:arf} that a $\Delta$-move necessarily changes the link type of a proper link. This is not the case for non-proper links such as the Hopf link; see Figure \ref{fig:hopf_delta}.
\begin{figure}[h]
    \centering  
    \includegraphics[width=0.5\textwidth]{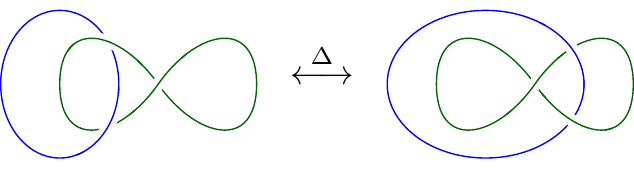}
    \caption{A Hopf link transformed into itself by a $\Delta$-move.}
    \label{fig:hopf_delta}
\end{figure}
\subsection{Milnor's Invariants}
\label{sec:milnor}

A self $\Delta$-move is a $\Delta$-move that only involves arcs from the same component of a link, as in Figure \ref{fig:delta_types}(a). We have the following classification of links up to self $\Delta$-equivalence:

\begin{corollary}[Corollary 1.5 in \cite{yasuhara2009self}]
    \label{cor:self_delta_classification}
    A link $L$ is self $\Delta$-equivalent to a trivial link if and only if $\bar{\mu}_L(I) = 0$ for any $I$ with $r(I) \leq 2$.
\end{corollary}

Here $\bar{\mu}_L(I)$ denotes Milnor's $\bar{\mu}$ invariants which measure the higher order linking of a link, introduced in \cite{milnor1954link, milnor1957isotopy}. For an $m$-component link, the multiindex $I = \{i_1 i_2 \cdots i_n\}$ takes values $1 \leq i_1,i_2,\ldots,i_n \leq m$, possibly repeated; $r(I)$ denotes the maximum number of times the indices $i_k$ repeats a value.

In particular, for 2-component algebraically split links, if $\bar{\mu}_L(1122) \neq 0$, then $L$ is not self $\Delta$-equivalent to the trivial link. It then follows from Proposition \ref{prop:self_delta_inequality} that $u^\Delta(L) > u^\Delta(L_1) + u^\Delta(L_2)$ and thus $u^{\Delta}(L) \geq u^{\Delta}(L_1) + u^{\Delta}(L_2) + 1$, improving the lower bound for $u^{\Delta} (L)$.

We can calculate $\bar{\mu}_L(1122)$ for a link $L$ using the link's Alexander polynomial.

\begin{theorem}[Theorem 2 in \cite{sturm1990arf}]
    \label{thm:mu_formula}
    A 2-component link $L$ has Alexander polynomial of the form $\Delta_L(x,y) = (x - 1)(y - 1) f(x,y)$ and
    \[|\bar{\mu}_L(1122)| = |f(1,1)|.\]
\end{theorem}

\begin{example}
    The link $L9a2$ has Alexander polynomial \cite{knotatlas}
    \[\Delta_{L9a2} (x,y) = \frac{(x-1)(y-1) \left(y^4 - y^3 + y^2 - y + 1\right)}{\sqrt{x} y^{5/2}}\]
    
    \noindent and thus $|\bar{\mu}_{L9a2}(1122)| = |f(1,1)| = 1$. Hence, $u^\Delta(L9a2)\geq u^\Delta(L9a2_1) + u^\Delta(L9a2_2) + 1 = 2$ since one of the components is an unknot and the other is a trefoil (which has $\Delta$-unknotting number 1). In fact, $u^\Delta(L9a2)=2$ since there exists the following $\Delta$-pathway:
    \[L9a2 \overset{\Delta}{\longleftrightarrow} 3_1 \# 0_1 \overset{\Delta}{\longleftrightarrow} 0^2_1.\]
\end{example}
\subsection{L9a18}
\label{sec:L9a18}

Some algebraically split links require additional methods to determine the $\Delta$-unlinking number. For instance, the link $L9a18$ can be transformed into the trivial link $0_1^2$ by three $\Delta$-moves via the pathway
\[L9a18 \overset{\Delta}{\longleftrightarrow} L7a4 \overset{\Delta}{\longleftrightarrow} L5a1 \overset{\Delta}{\longleftrightarrow} 0_1^2.\]

Moreover, since $\arf(L9a18)=1$, we conclude from Corollary \ref{cor:arf} that $u^\Delta(L9a18)$ is 1 or 3.

Suppose $u^\Delta(L9a18) = 1$. Since $|\bar{\mu}_{L9a18}(1122)| = 3 \neq 0$, by Corollary \ref{cor:self_delta_classification} we know $L9a18$ is not self $\Delta$-equivalent to the trivial link. Thus the $\Delta$-move must contain two strands belonging to one component of $L9a18$ and one distinguished strand belonging to the other. As $L9a18$ is invertible, we may deform the diagram via ambient isotopy such that the distinguished strand belongs to a component that is an unknotted circle in the diagram.

\begin{figure}[b]
    \centering  
    \includegraphics[width=0.7\textwidth]{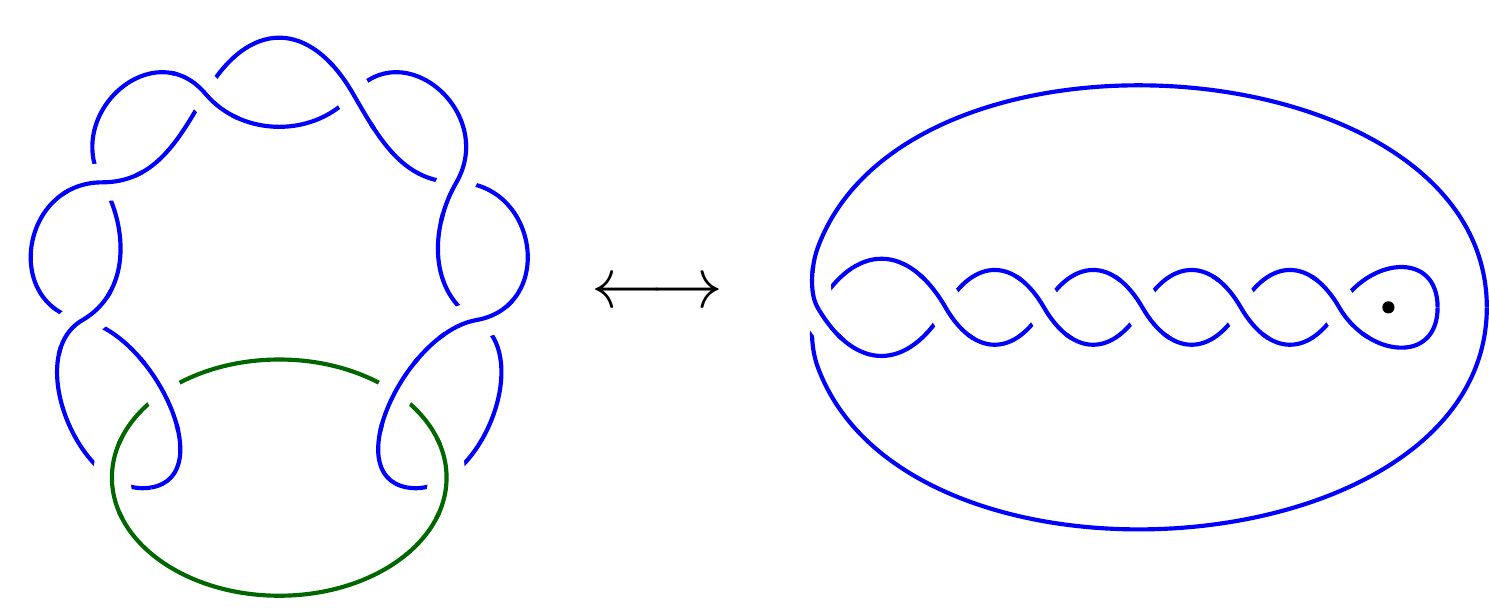}
    \caption{The link $L9a18$ can be represented as a knot in a solid 1-torus.}
    \label{fig:L9a18}
\end{figure}

When a link $L$ has an unknotted circular component, it can be represented as a knot $K_L$ on a punctured diagram, or equivalently a knot in the solid torus; see Figure \ref{fig:L9a18}.

Call a $\Delta$-move a toroidal $\Delta$-move if one arc belongs to an unknotted circular component and the other two arcs belong to the other component of a 2-component link. Then, the toroidal $\Delta$-move has a corresponding move in the punctured diagram or solid torus as depicted in Figure \ref{fig:torodial_delta_move}.

There is a simple numerical invariant of a knot $K$ in the solid torus, denoted $\beta_1(K)$, defined by lifting $K$ to its infinite cyclic cover and calculating the linking number $lk(K_0,K_1)$ \cite{bataineh2015numerical}. Figure \ref{fig:L9a18_cyclic_covering} shows the lift for $K_{L9a18}$ from which we determine $\left|\beta_1(K_{L9a18})\right|=3$.

Since a toroidal $\Delta$-move changes two crossings, it will change $\beta_1$ by at most 2. Hence, as the trivial link has vanishing $\beta_1$, $K_{L9a18}$ cannot be one toroidal $\Delta$-move away from the trivial link. Hence $u^\Delta(L9a18)\neq 1$ so we conclude $u^\Delta(L9a18)=3$.

\begin{figure}[t]
    \centering
    \includegraphics[width=0.7\textwidth]{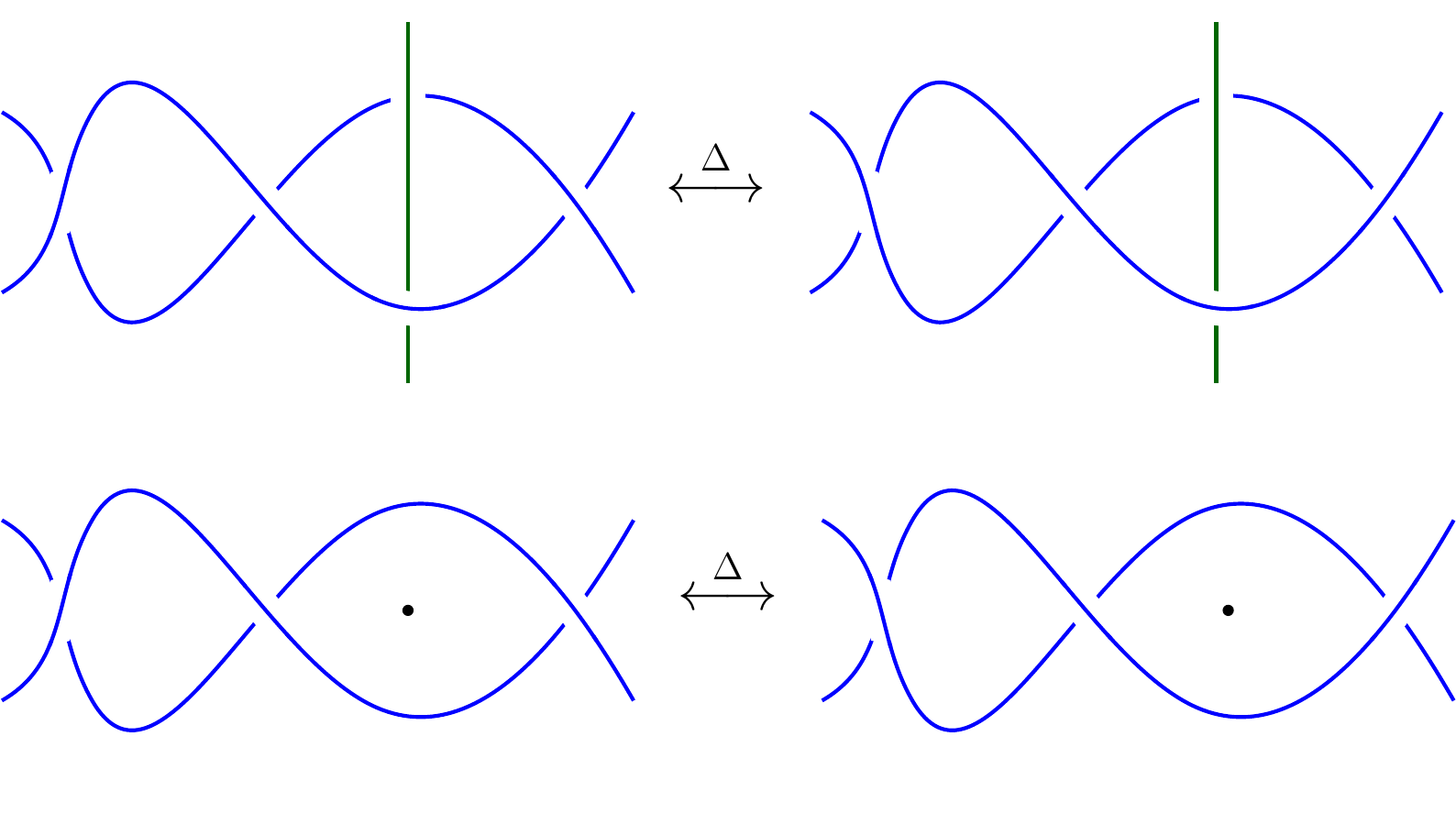}
    \caption{A $\Delta$-move in a punctured diagram.}
    \label{fig:torodial_delta_move}
\end{figure}

\begin{figure}[h]
    \centering
    \includegraphics[width=0.8\textwidth]{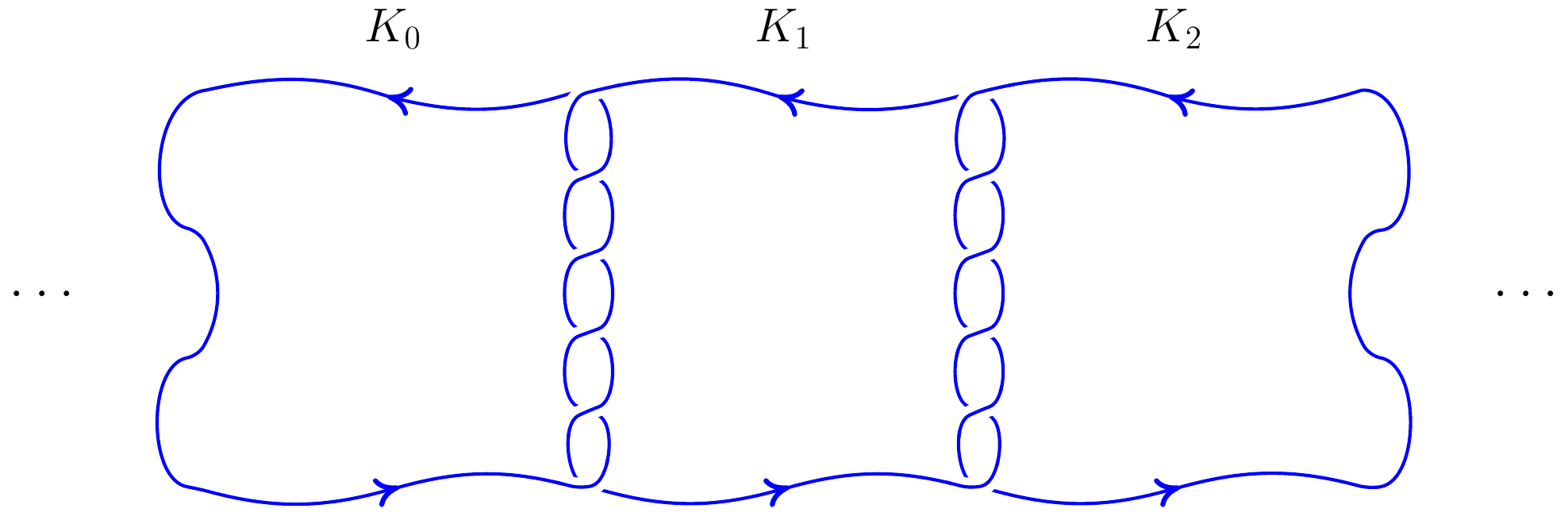}
    \caption{The infinite cyclic cover of $K_{L9a18}$.}
    \label{fig:L9a18_cyclic_covering}
\end{figure}
\section{Table of $\Delta$-unlinking numbers}
\label{sec:table}

We tabulate here the algebraically split prime links with their $\Delta$-unlinking numbers. The Arf invariants are from \cite{montemayor2012nullification}. The unlinking numbers are from \cite{nagel2015unlinking}. The Rolfsen names are from Knot Atlas \cite{knotatlas}. The 4-genus lower bound was calculated using Corollary 1.5 in \cite{powell2017four}. We determined the exact value of the 4-genus of many of the links by band summing to construct explicit upper bounds. The $\bar{\mu}(1122)$ invariant was calculated using Theorem \ref{thm:mu_formula}. The column headers are consistent with the notation in the text, but for clarity we have, in order: link name (Thistlethwaite and Rolfsen), $\Delta$-unlinking number, half of the unlinking number, sum of delta-unknotting numbers, Arf invariant, 4-genus, Milnor $\bar{\mu}$ invariant, and the method(s) used to calculate the delta-unlinking number.

\begin{table}
    \centering
    \caption{$\Delta$-unlinking number and certain invariants for algebraically split links up to 9 crossings.}
    \begin{tabular}{c c c c c c c c c c}
        \toprule
        \multicolumn{2}{c}{Link} & \multirow{2}{*}{$u^\Delta(L)$} & \multirow{2}{*}{$\frac{1}{2} u(L)$} & \multirow{2}{*}{$\sum u^\Delta(L_i)$} & \multirow{2}{*}{$\arf(L)$} & \multirow{2}{*}{$g_4(L)$} & \multirow{2}{*}{$|\bar{\mu}_L(1122)|$} & \multirow{2}{*}{Method(s)} \\
        Thistlethwaite & Rolfsen & & & & & & & \\
        \midrule
        $L5a1$ & $5_2^1$ & 1 & 0.5 & 0 & 1 & 1 & 1 & \small{Prop} \ref{prop:half_unlinking_number} \\
        $L6a4$ & $6_3^2$ & 1 & 1 & 0 & 1 & 1 & -- & \small{Prop} \ref{prop:half_unlinking_number} \\
        $L7a1$ & $7_2^6$ & 1 & 1 & 0 & 1 & 1 & 1 & \small{Prop} \ref{prop:half_unlinking_number} \\
        $L7a3$ & $7_2^4$ & 3 & 1 & 1 & 1 & 2 & 2 & \small{Cor} \ref{cor:4_genus_links_unlinking}, \small{Cor} \ref{cor:arf} \\
        $L7a4$ & $7_2^3$ & 2 & 1 & 0 & 0 & 1 & 2 & \small{Prop} \ref{prop:half_unlinking_number}, \small{Cor} \ref{cor:arf} \\
        $L7n2$ & $7_2^8$ & 2 & 0.5 & 1 & 0 & 1 & 1 & \small{Prop} \ref{prop:half_unlinking_number}, \small{Cor} \ref{cor:arf} \\
        $L8a1$ & $8_2^{13}$ & 1 & 1 & 0 & 1 & 1 & 1 & \small{Prop} \ref{prop:half_unlinking_number} \\
        $L8a2$ & $8_2^{10}$ & 1 & 0.5 & 1 & 1 & 1 & 0 & \small{Prop} \ref{prop:half_unlinking_number} \\
        $L8a4$ & $8_2^{12}$ & 1 & 0.5 & 1 & 1 & 1 & 0 & \small{Prop} \ref{prop:half_unlinking_number} \\
        $L8n2$ & $8_2^{15}$ & 2 & 0.5 & 1 & 0 & 1 & 1 & \small{Prop} \ref{prop:half_unlinking_number}, \small{Cor} \ref{cor:arf} \\
        $L9a1$ & $9_2^{32}$ & 1 & 1 & 0 & 1 & 1 & 1 & \small{Prop} \ref{prop:half_unlinking_number} \\
        $L9a2$ & $9_2^{31}$ & 2 & 1.5 & 1 & 0 & 2 & 1 & \small{Prop} \ref{prop:half_unlinking_number} \\
        $L9a3$ & $9_2^{33}$ & 2 & 1 & 1 & 0 & 1 & 1 & \small{Prop} \ref{prop:half_unlinking_number}, \small{Cor} \ref{cor:arf} \\
        $L9a4$ & $9_2^{18}$ & 4 & 1 & 2 & 0 & 2 & 2 & \small{Cor} \ref{cor:arf}, \small{Sec} \ref{sec:milnor} \\
        $L9a8$ & $9_2^{25}$ & 3 & 1 & 1 & 1 & 1 or 2 & 2 & \small{Cor} \ref{cor:arf}, \small{Sec} \ref{sec:milnor} \\
        $L9a9$ & $9_2^{37}$ & 2 & 1 & 0 & 0 & 1 or 2 & 2 & \small{Prop} \ref{prop:half_unlinking_number}, \small{Cor} \ref{cor:arf} \\
        $L9a10$ & $9_2^{36}$ & 3 & 1.5 & 2 & 1 & 1 or 2 & 2 & \small{Prop} \ref{prop:half_unlinking_number}, \small{Cor} \ref{cor:arf} \\
        $L9a14$ & $9_2^{13}$ & 4 or 6 & 1.5 & 3 & 0 & 3 & 3 & \small{Cor} \ref{cor:4_genus_links_unlinking}, \small{Cor} \ref{cor:arf} \\
        $L9a15$ & $9_2^{15}$ & 3 or 5 & 1.5 & 2 & 1 & 2 & 3 & \small{Cor} \ref{cor:4_genus_links_unlinking}, \small{Cor} \ref{cor:arf} \\
        $L9a17$ & $9_2^{27}$ & 2 or 4 & 1.5 & 1 & 0 & 2 & 3 & \small{Cor} \ref{cor:4_genus_links_unlinking}, \small{Cor} \ref{cor:arf} \\
        $L9a18$ & $9_2^{10}$ & 3 & 1 & 0 & 1 & 1 & 3 & Sec \ref{sec:L9a18} \\
        $L9a35$ & $9_2^9$ & 1 & 1 & 0 & 1 & 1 & 3 & \small{Prop} \ref{prop:half_unlinking_number} \\
        $L9a38$ & $9_2^5$ & 2 & 0.5 & 0 & 0 & 1 or 2 & 4 & \small{Prop} \ref{prop:half_unlinking_number}, \small{Cor} \ref{cor:arf} \\
        $L9a40$ & $9_2^4$ & 3 & 1 & 0 & 1 & 2 & 5 & \small{Cor} \ref{cor:4_genus_links_unlinking}, \small{Cor} \ref{cor:arf} \\
        $L9a42$ & $9_2^{41}$ & 1 & 1 & 0 & 1 & 1 & 3 & \small{Prop} \ref{prop:half_unlinking_number} \\
        $L9a53$ & $9_3^{12}$ & 1 & 1 & 0 & 1 & 1 & -- & \small{Prop} \ref{prop:half_unlinking_number} \\
        $L9a54$ & $9_3^9$ & 3 & 1.5 & 0 & 1 & 2 or 3 & -- & \small{Prop} \ref{prop:half_unlinking_number}, \small{Cor} \ref{cor:arf} \\
        $L9n2$ & $9_2^{46}$ & 4 & 1 & 2 & 0 & 1 & 2 & \small{Cor} \ref{cor:arf}, \small{Sec} \ref{sec:milnor} \\
        $L9n3$ & $9_2^{47}$ & 3 & 0.5 & 2 & 1 & 1 & 1 & \small{Prop} \ref{prop:self_delta_inequality}, \small{Cor} \ref{cor:arf} \\
        $L9n5$ & $9_2^{44}$ & 5 & 1 & 3 & 1 & 2 & 2 & \small{Cor} \ref{cor:arf}, \small{Sec} \ref{sec:milnor} \\
        $L9n6$ & $9_2^{55}$ & 4 & 1 & 3 & 0 & 2 & 1 & \small{Prop}  \ref{prop:self_delta_inequality}, \small{Cor} \ref{cor:arf} \\
        $L9n8$ & $9_2^{56}$ & 3 & 1 & 2 & 1 & 2 & 1 & \small{Prop}  \ref{prop:self_delta_inequality}, \small{Cor} \ref{cor:arf} \\
        $L9n25$ & $9_3^{18}$ & 2 & 1 & 0 & 0 & 1 & -- &  \small{Prop}  \ref{prop:half_unlinking_number}, \small{Cor} \ref{cor:arf} \\
        $L9n27$ & $9_3^{21}$ & 2 & 0.5 & 0 & 0 & 2 & -- &  \small{Prop}  \ref{prop:half_unlinking_number}, \small{Cor} \ref{cor:arf} \\
        \bottomrule
    \end{tabular}
    \label{tab:delta_unlinking}
\end{table}

\clearpage
We can also display the {\it minimal $\Delta$-pathways} as a tree. In Figure \ref{fig:link_graph} each edge represents one $\Delta$-move and the path from a link $L$ to the trivial link is a pathway of minimal length $u^\Delta(L)$. Indeterminate cases are represented by a dashed edge. Recall $L \# L'$ denotes a split union between $L$ and $L'$.

\begin{figure}[h]
    \centering
    \includegraphics[width=\textwidth]{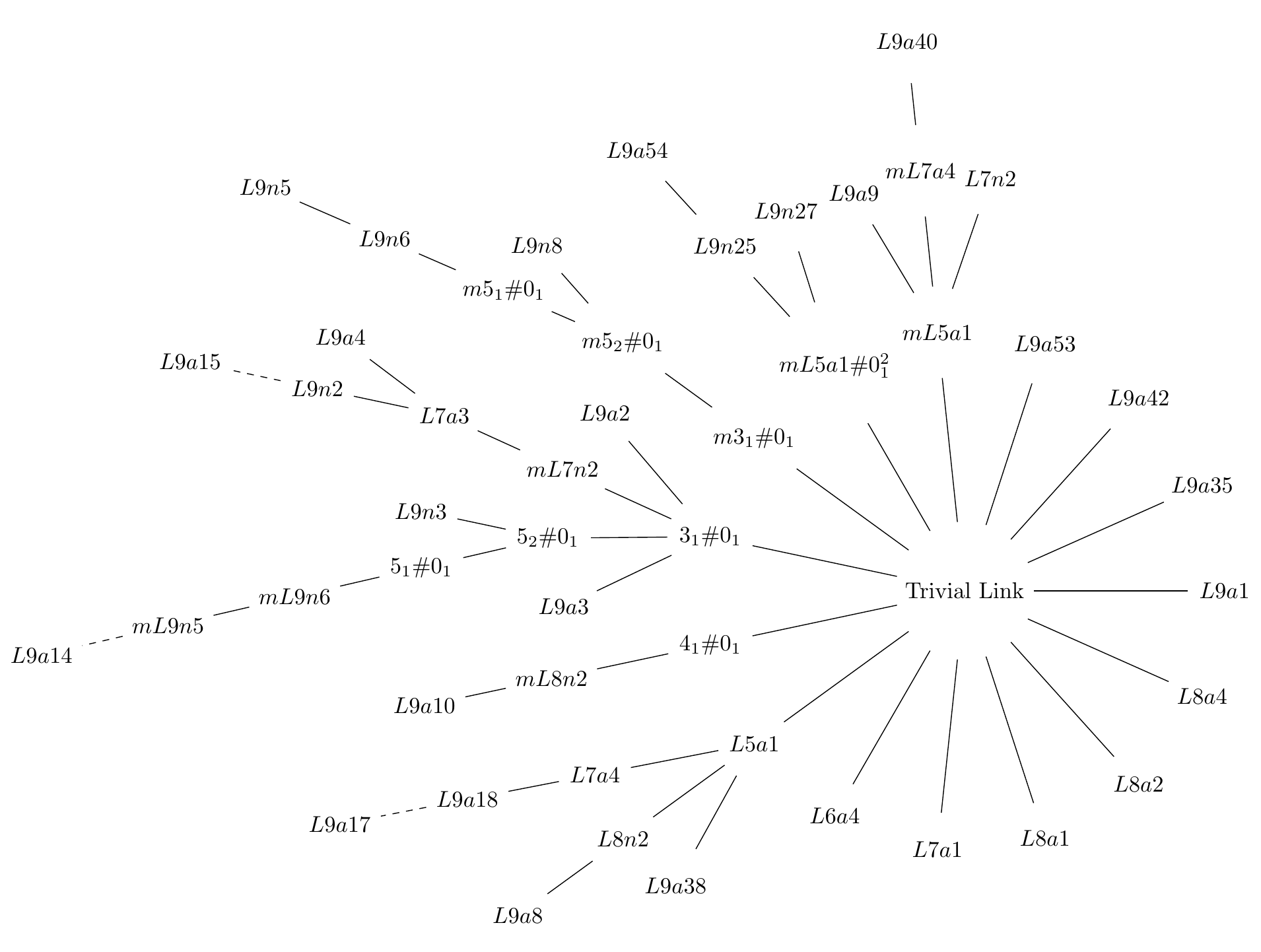}
    \caption{A tree exhibiting $\Delta$-pathways of minimal length to the trivial link.}
    \label{fig:link_graph}
\end{figure}

\clearpage
\printbibliography

\end{document}